\documentclass[11pt]{article}

\usepackage{graphicx, amssymb, latexsym, amsfonts, amsmath, lscape, amscd,
amsthm, color, epsfig, mathrsfs, tikz, enumerate}

\newtheorem{theorem}{Theorem}[section]

\newtheorem{lemma}[theorem]{Lemma}

\newcommand\DELETE[1]{}

\begin{document}

\title{{\bf Chromatic number of signed graphs  with bounded maximum degree}}
\author{
{\sc Sandip Das}$\,^a$, {\sc Soumen Nandi}$\,^a$, {\sc Soumyajit Paul}$\,^{b}$, {\sc Sagnik Sen}$^{a}$\\
\mbox{}\\
{\small $(a)$ Indian Statistical Institute, Kolkata, India}\\
{\small $(b)$ Chennai Mathematical Institute, Chennai, India}
}

\date{\today}

\maketitle

\begin{abstract}
A   signed graph $ (G, \Sigma)$ is a graph  positive 
and negative ($\Sigma $ denotes the set of negative edges).
To re-sign  a vertex $v$ of a signed graph $ (G, \Sigma)$ 
is to switch the signs of the edges incident to $v$.  
If one can obtain $ (G, \Sigma')$ by re-signing some vertices of $ (G, \Sigma)$, then
 $ (G, \Sigma) \equiv (G, \Sigma')$.
 A signed graphs  $ (G, \Sigma )$ admits an homomorphism to  
$ (H, \Lambda )$ if there is a sign preserving vertex mapping from  $(G,\Sigma')$ to $(H, \Lambda)$
for some $ (G, \Sigma) \equiv (G, \Sigma')$.
    The  signed chromatic number $\chi_{s}( (G, \Sigma))$ of the  signed graph $(G, \Sigma)$ is the minimum order (number of vertices)  of a  signed graph $(H, \Lambda)$ such that
$ (G, \Sigma)$ admits a homomorphism to   $(H, \Lambda)$. 
For a family  $ \mathcal{F}$ of signed graphs
$\chi_{s}(\mathcal{F}) = \text{max}_{(G,\Sigma) \in \mathcal{F}} \chi_{s}( (G, \Sigma))$.
 We prove $2^{\Delta/2-1} \leq \chi_s(\mathcal{G}_{\Delta}) \leq (\Delta-1)^2. 2^{(\Delta-1)} +2$ for all $\Delta \geq 3$ where $\mathcal{G}_{\Delta}$ is the family of connected signed graphs with maximum degree $\Delta$.
\end{abstract}


\section{Introduction and the main result}
A   \textit{signed graph} $ (G, \Sigma)$ is a graph $G$ with  \textit{positive} 
and \textit{negative} edges where $\Sigma $ denotes the set of negative edges and $G$ denotes the underlying graph. 
The set of positive edges is denoted by $\Sigma^c$. 
When the set of negative edges $\Sigma$ is known from the context, we can denote the signed graph $(G,\Sigma)$ by $(G)$.
In general, the set of vertices and the set of edges of the signed graph $(G, \Sigma)$ are
 denoted by $V(G)$ and 
$E(G)$. 
To \textit{re-sign}  a vertex $v$ of a signed graph $ (G, \Sigma)$ 
is to switch the signs of the edges incident to $v$.  
Two  signed graphs $ (G, \Sigma)$ and $ (G, \Sigma^\prime)$ are  \textit{equivalent} 
if we can obtain $ (G, \Sigma^\prime)$ by re-signing some vertices of 
$ (G, \Sigma)$. 
Given a signed graph an adjacent vertex of $v$ is called its \textit{neighbor}. 
The set of all neighbors of $v$ is denoted by $N(v)$ 
while $d(v) = |N(v)|$ is the \textit{degree} of $v$.

A \textit{2-edge-colored homomorphism} $\psi$ of  $ (G, \Sigma )$ to
$ (H, \Lambda )$ is a vertex mapping 
$ \psi : V(G) \longrightarrow V(H)$ 
such that  
   for each edge $uv$ of $(G,\Sigma)$  the images induces an edge  $\phi(u)\phi(v)$ in $(H, \Lambda)$ of the same sign as 
 $uv$.
 We write $ (G, \Sigma) \xrightarrow{2ec}  (H, \Lambda)$ whenever there exists a 
 2-edge-colored homomorphism of $ (G, \Sigma)$ to $ (H, \Lambda)$.

Given two signed graphs  $ (G, \Sigma )$ and 
$ (H, \Lambda )$, 
$ \phi : V(G) \longrightarrow V(H)$ is a  \textit{homomorphism} of 
 $(G, \Sigma)$ to 
 $(H, \Lambda)$ if 
there exists 
a $(G,\Sigma')$ equivalent to $(G, \Sigma)$  
such that  
   $\phi$ is a 2-edge-colored homomorphism of $(G,\Sigma')$ to $(H, \Lambda)$.
    We write $ (G, \Sigma) \rightarrow  (H, \Lambda)$ whenever there exists a 
  homomorphism of $ (G, \Sigma)$ to $ (H, \Lambda)$.

 The  \textit{signed chromatic number} $\chi_{s}( (G, \Sigma))$ or the \textit{2-edge-colored chromatic number}
 $\chi_{2}( (G, \Sigma))$  of the  signed graph $(G, \Sigma)$ is the minimum \textit{order} (number of vertices)  of a  signed graph $(H, \Lambda)$ such that
$ (G, \Sigma) \rightarrow  (H, \Lambda)$ or $ (G, \Sigma) \xrightarrow{2ec}  (H, \Lambda)$, respectively. 
The signed chromatic number 
$\chi_{s}(\mathcal{F})$ or the 2-edge-colored chromatic number  $\chi_{2}( \mathcal{F})$ of  a family $ \mathcal{F}$ of signed graphs is the maximum of the signed chromatic numbers 
 or the 2-edge-colored chromatic numbers, respectively,  of the signed graphs from the family $\mathcal{F}$.

The signed graphs and their switch classes have been studied since the begining of the last century~\cite{harary, Zaslavsky} while the homomorphism 
of signed graphs have been introduced and studied recently by Naserasr, Rollova and Sopena~\cite{signedhom}.
 Till now, the relation between the maximum 
degree of a signed graph and its chromatic number is not studied. We initiate it by proving the
following result adapting a probabilistic proof technique used by Kostochka, Sopena and Zhu~\cite{Kostochka97acyclicand}. 

\begin{theorem}\label{th main}
If  $\mathcal{G}_{\Delta}$ is the family of signed graphs with maximum degree at most $\Delta $, 
then $2^{\Delta/2 - 1} \leq \chi_s(\mathcal{G}_{\Delta}) \leq (\Delta-1)^2. 2^{(\Delta-1)} +2$ for all $\Delta \geq 3$.
\end{theorem}

Note that both lower and upper bounds are exponential in $\Delta$.  

\section{Proof of Theorem~\ref{th main}}
First we will prove the lower bound.

Let $(G, \Sigma)$ be a signed graph. 
Let $uv$ be an positive edge and $xy$ be a negative edge. Then $u$ is
a $+$-neighbor of $v$ and $x$ is a $-$-neighbor of $y$. The set of all $+$-neighbors and $-$-neighbors of a
vertex $v$ is denoted by $N^+(v)$ and $N^-(v)$, respectively.
Let $\vec{a} = (a_1, a_2, ..., a_j)$ be a \textit{$j$-vector} such that 
$a_i \in \{+,-\}$ where $i \in \{1,2,...,j\}$. 
Let $J = (v_1, v_2, ..., v_j)$ be a \textit{$j$-tuple} (without repetition) of  vertices  from $G$. Then we define the set 
$N^{\vec{a}}(J) = \{v \in V | v \in N^{a_i}(v_i) \text{ \textit{for all} } 1 \leq i \leq j\} $. 
Finally, we say that $G$ has property $P_{t-1}$ if for each $j$-vector $\vec{a}$ and each $j$-tuple $J$ we have 
$|N^{\vec{a}}(J)| \geq \frac{1+(t-j)(t-2)}{2}$ where $j \in \{0,1,...,t-1\}$.

\begin{lemma}\label{key-lemma}
There exists a signed  complete graph  with property $P_{t-1}$ on 
$c =  t(t-1) .2^{t}$ vertices. 
\end{lemma}

\begin{proof}
Let $(C, \Pi)$ be a random signed graph with underlying complete graph. 
Let $u,v$ be two vertices of $(C, \Pi)$ and the events $u \in N^{a}(v)$
for $a \in \{+,-\}$ are equiprobable and independent  with probability $\frac{1}{2}$. 
We will show that the probability of $C$ not having property $P_{t-1}$ is strictly less than 1 when  
$|C| = c = t(t-1) .2^{t}$. Let $P(J,\vec{a})$ denote the probability of the event 
$|N^{\vec{a}}(J)| < \frac{(t-j)(t-2)+1}{2}$ where $J$ is a $j$-tuple of $C$ and $\vec{a}$ is a $j$-vector for some $j \in \{0,1,...,t-1\}$. Call such an event a \textit{bad event}. Thus,

\begin{equation} \label{eq1}
\begin{split}
P(J, \vec{a})  \leq \sum\limits_{i=0}^{\frac{(t-j)(t-2)-1}{2}} {c-j \choose i} 2^{-ij} (1 - 2^{-j})^{c - i - j} & < (1 - 2^{-j})^c \sum\limits_{i=0}^{\frac{(t-j)(t-2)-1}{2}} \frac{c^i}{i!} (1 - 2^{-j})^{- i - j} 2^{-ij} \\
 & < 2 e^{-c2^{-j}} \sum\limits_{i=0}^{\frac{(t-j)(t-2)-1}{2}} c^i  < e^{-c2^{-j}} c^{\frac{(t-j)(t-2)+1}{2}}
\end{split}
\end{equation}

Let $P(B)$ denote the probability of the occurrence of at least one bad event. 
To prove this lemma it is enough to  show that $P(B) < 1$. Let $T^j$ denote the set of all $j$-tuples and $W^j$ denote the set of all $j$-vectors.  Then 

\begin{equation} \label{eq2}
\begin{split}
P(B) & = \sum_{j=0}^{t - 1} \sum_{J \in T^j} \sum_{\vec{a} \in W^j} P(J, \vec{a})  < \sum\limits_{j=0}^{t - 1} {c \choose j} 2^{j} e^{-c2^{-j}} c^{\frac{(t-j)(t-2)+1}{2}} \\
 & < 2\sum\limits_{j=0}^{t - 1} \frac{c^j}{j!} 2^{j-1} e^{-c2^{-j}} c^{\frac{(t-j)(t-2)+1}{2}}  <  2\sum\limits_{j=0}^{t - 1} e^{-c2^{-j}} c^{\frac{(t-j)(t-2)+1}{2}+j}
\end{split}
\end{equation}

Consider the function $f(j) =  2e^{-c2^{-j}} c^{\frac{(t-j)(t-2)+1}{2} + j}$. Observe that $f(j)$ is the $j^{th}$ 
summand  of the last sum from equation~(\ref{eq2}). 
Now
\begin{equation} \label{eq3}
\begin{split}
 \frac{f(j + 1)}{f(j)} & =  \frac{e^{c2^{-j-1}}}{c^{\frac{t-4}{2}}}  >   \frac{e^{c2^{-t+1}}}{c^{\frac{t-4}{2}}}  =   \frac{e^{2t(t-1)}}{(t(t-1))^{\frac{t-4}{2}}2^{\frac{t(t-4)}{2}}}  >   \left(\frac{e^{2(t-1)}}{t(t-1)} \cdot \frac{e^{2(t-1)}}{2^{t-4}}\right)^{\frac{t}{2}}  >   2^{10} 
\end{split}
\end{equation}

\medskip

The above relation implies the following

\begin{equation} \label{eq6}
\begin{split}
 P(B) & \leq  \sum\limits_{j=0}^{t - 1} f(j) < \left(1+\frac{1}{2^{10}-1}\right) f(t-1)  < 2(1.001) \left( \frac{t^3(t-1)^3}{e^{2t-1}}\cdot \frac{2^{3t}}{e^{2t-1}}  \right)^{\frac{t-1}{2}}  <   1 
\end{split}
\end{equation}

This completes the proof.
\end{proof}

\begin{lemma}\label{key-lemma2}
Let $(C, \Pi)$ be a signed graph 
 with property 
$P_{\Delta-1}$ and $(G, \Sigma)$ be a connected signed graph with 
maximum degree $\Delta$ and degeneracy $(\Delta-1)$. Then $(G, \Sigma)$ admits a homomorphism to $(C, \Pi)$. 
\end{lemma}

\begin{proof}
Let $v_1, v_2, ..., v_k$ be the vertices of $(G,\Sigma)$ in such a way that each vertex $v_j$ has at most $(\Delta-1)$ neighbors with lower indices.
Let $(G_j, \Sigma_j)$ be the  signed graph induced by the vertices $v_1, v_2, ..., v_l$ from $(G,\Sigma)$ for each $j \in \{1,2,...,k\}$. 
Now we will inductively construct a  homomorphism $f: (G,\Sigma) \rightarrow (C,\Pi)$ with the following properties:

\begin{itemize}
\item[$(i)$] The partial mapping $f(v_1), f(v_2), ..., f(v_j)$ is a homomorphism of 
 $(G_j, \Sigma_j)$ to $(C, \Pi)$ for all $j \in \{1,2,...,k\}$.

\item[$(ii)$]  For each $i > j$, all the neighbors of $v_i$ with indices less than or equal to $j$ has different images with respect to the mapping $f$.  
\end{itemize} 

For $j=1$ take any  partial mapping $f(v_1)$. 
Suppose that the function $f$ satisfies the above properties for all $i \leq j$ for some  fixed $j \in \{1,2,...,k-1\}$. 
Let $A$ be the set of neighbors of $v_{j+1}$   with  indices greater than $j+1$
and $B$ be the set of vertices   with indices at most $j$ and with at least one neighbor in $A$. 
 Note that  $|B| = (\Delta-2)|A|$. 
 Let $D$ be the set of possible  options for $f(v_{j+1})$ such that the partial mapping 
 is a homomorphism of $(G_{j+1},\Sigma_{j+1})$ to $(C,\Pi)$. As the partial mapping can be 
 extended also by re-signing the vertex $v_{j+1}$. Thus 
  $|D| > |B|$.
 Choose any vertex from $D \setminus B$ as the image $f(v_{t+1})$. 
 Note that this partial mapping satisfies the required conditions. 
\end{proof}

\medskip

\noindent \textit{Proof of Theorem~\ref{th main}.}  
The lower bound proof of oriented chromatic number for $\mathcal{G}_{\Delta}$ by Kostochka, Sopena and Zhu~\cite{Kostochka97acyclicand}
can be easily modified to obtain the lower bound   $2^{\Delta/2} \leq \chi_2(\mathcal{G}_{\Delta})$. 
Our lower bound follows from the relation  
$\chi_2((G,\Sigma)) \leq 2 \cdot \chi_s((G,\Sigma))$ for any signed graph $(G,\Sigma)$~\cite{signedhom}.

\medskip
For proving the upper bound, if $(G, \Sigma)$ is not $\Delta$-regular, but is a connected signed graph with maximum degree $\Delta$, then $(G, \Sigma)$ is $\Delta-1$ degenerated
then we are done by Lemma~\ref{key-lemma} and~\ref{key-lemma2}. 
Otherwise,  $(G, \Sigma)$ is a $\Delta$-regular connected signed graph. 
Delete one edge $uv$ from $(G, \Sigma)$ to obtain a connected signed graph which has maximum
degree $\Delta$ and degeneracy $\Delta-1$. This new signed graph admits a  homomorphism $g$ to a signed graph $(C, \Pi)$   with property 
$P_{\Delta-1}$. Now modify this homomorphism by changing the 
images $g(u)$ and $g(v)$ by adding two new vertices $g(u), g(v)$   to $C$ and choosing signs of the edges $uv$ and the edges between $\{u,v\}$ and the vertices of $C$ in such a way that our newly obtained map is also a homomorphism. 
\hfill $ \square$

   \section{Conclusive remarks}
  Klostermeyer and MacGillivray~\cite{push} studied pushable chromatic number of oriented graphs. 
  The exact same upper and lower bounds proved in Theorem~\ref{th main} can be proved for pushable chromatic number of connected graphs with bounded maximum degree in a similar way.

\bibliographystyle{plain}
\bibliography{NSS14}

\end{document}